\documentclass[12pt]{amsart}
\usepackage{amssymb, amsmath, amscd, amsthm, setspace, latexsym, vruler, enumerate}
\usepackage{graphicx}
\usepackage{mathtools}
\usepackage[toc,page]{appendix}
\title{Some properties of coefficients of cyclotomic polynomials}
\author{Marcin Mazur}
\author{Bogdan V.~Petrenko}

\address{
Department of Mathematics \\
Binghamton University \\
P.O. Box 6000 \\
Binghamton, NY 13892-6000, USA } \email{
mazur@math.binghamton.edu}

\address{
Department of Mathematics and Computer Science\\ Eastern Illinois University \\ 
600 Lincoln Avenue \\ Charleston, IL 61920-3099, USA
 }

\email{ bvpetrenko@eiu.edu}

\newtheorem{theorem}{Theorem}[section]
\newtheorem{lemma}[theorem]{Lemma}

\newtheorem{corollary}[theorem]{Corollary}

\newtheorem{conjecture}[theorem]{Conjecture}

\def\d{\displaystyle}

\newtheorem*{te*}{Theorem}

\begin{document}
\maketitle

\begin{abstract} 

This paper investigates coefficients
of cyclotomic polynomials theoretically and experimentally.
We prove the following result. {{\em If $n=p_1\ldots p_k$ where $p_i$ are odd primes and
	$p_1<p_2<\ldots<p_r<p_1+p_2<p_{r+1}<\ldots<p_t$ with $t\geq 3$ odd,
	then the numbers $-(r-2),-(r-3),\ldots, r-2, r-1$ are all coefficients
	of the cyclotomic polynomial $\Phi_{2n}$. Furthermore, if $1+p_r<p_1+p_2$ then $1-r$ is also a coefficient of $\Phi_{2n}$.} In the experimental part, in two instances we present computational evidence for asymptotic symmetry between distribution of positive and negative coefficients, and state the resulting conjectures.}

\vspace{4mm} \noindent {\bf Mathematics Subject Classification
(2010).} Primary: 11C08. Secondary: 11B05, 11B83.  

\vspace{3mm} \noindent {\bf Keywords.} Cyclotomic polynomials, coefficients, distribution, symmetry. 
\end{abstract}

\section{Introduction}
Cyclotomic polynomials $\Phi_n(x)$ can be defined by induction as 
follows: $\Phi_1(x)=x-1$, and subsequently
$\Phi_n(x)$ is the quotient of $x^n-1$ by the product of $\Phi_d(x)$
taken over all $d<n$ that divide $n$.
The polynomial $\Phi_n(x)$ is the minimal polynomial over $\mathbb Q$ of a primitive $n$th root of 1.

There has has been a considerable research on coefficients of 
cyclotomic polynomials. The first cyclotomic polynomial with a {\sffamily{nontrivial}} coefficient (a coefficient different from 0 and $\pm 1$) 
is $\Phi_{105}(x)$, a fact established in \cite{mig}. Many leading mathematicians have subsequently studied 
the coefficients of $\Phi_n(x)$. In 1936, Emma Lehmer \cite{lehmer} included the proof  
by Issai Schur that there exist cyclotomic polynomials with coefficients 
arbitrarily large in absolute value. In 1946, Paul Erd\"{o}s \cite{erdos} showed that there 
is $c>0$ such that for infinitely many $n$ the absolute value of the largest 
coefficient of $\Phi_n(x)$ is at least
$\d \exp \left\{ c (\log n)^{4/3} \right\}$. In 1949, Paul Bateman \cite{bateman} proved that there 
is $d>0$ such that for infinitely many $n$, the absolute value of the largest 
coefficient is at most $\d  < \exp \left\{ n^{d/ \log \log n} \right\}$. 
Surprisingly, only in 1987 Jiro Suzuki \cite{suzuki} (by improving the argument of 
Schur) showed that any integer is a coefficient of some cyclotomic polynomial.

Despite the above results, the family of cyclotomic polynomials having a given 
integer as a coefficient is 
mysterious for the most part. In particular, we do not seem to 
know anything about the value of the smallest degree of polynomials
in this family. The present paper grew out of our attempts to gain some insight into 
this problem. In Section 2, Theorem \ref{main}, we prove the following result which strengthens the 
aforementioned result of Suzuki.

\vspace{3mm}
\noindent
{\bf Theorem.}
{\em If $n=p_1\ldots p_k$ where $p_i$ are odd primes and
$p_1<p_2<\ldots<p_r<p_1+p_2<p_{r+1}<\ldots<p_t$ with $t\geq 3$ odd,
then the numbers $-(r-2),-(r-3),\ldots, r-2, r-1$ are all coefficients
of $\Phi_{2n}$. Furthermore, if $1+p_r<p_1+p_2$ then $1-r$ is also a coefficient of $\Phi_{2n}$.}

\vspace{3mm}

Section 3 contains some numeric data about the distribution of nontrivial coefficients of cyclotomic polynomials. We have 
observed some puzzling symmetry between the distribution of positive and 
negative coefficients. The meaning of this symmetry should become clear once 
the reader sees the graphs included in Section 3 as well as Conjectures \ref{conj1} and \ref{conj2}. At present we do not have any 
mathematical framework to explain what the pictures obviously suggest. The 
data were obtained by Brett Haines with the help of Wolfram
Mathematica \cite{wolf} and William Tyler Reynolds with the help of SAGE \cite{sage} as part of their 
independent studies with Bogdan 
Petrenko when they were undergraduate students at EIU.

\section{Proof of Theorem}

By $\Phi_n=\Phi_n(x)$ we denote the $n$-th cyclotomic polynomial. 
Let $\xi_n=e^{2\pi i/n}$, so $\xi_n$ is a primitive $n$-th root of $1$.
We define $\Pi_n$ to be the set of all primitive $n$-th roots of $1$, i.e.
\[\Pi_n=\{\xi_n^a:1\leq a\leq n \ \text{and}\ \gcd(a,n)=1\}\]
Let $S_k(n)=\sum_{u\in\Pi_n} u^{k}$.  Furthermore, let $\sigma_k(n)$ be the coefficient
of $\Phi_n$ at $x^{\phi(n)-k}$, where $\phi$ is the Euler function (so $\phi(n)$ is the degree
of $\Phi_n$). Thus $(-1)^k\sigma_k(n)$ is the value of the $k$-th elementary
symmetric function in $\phi(n)$ variables evaluated at the primitive $n$-th roots of $1$. 

\begin{lemma}
$S_1(n)=\mu(n)$, where $\mu$ is the M\"obius function.
\end{lemma}

\begin{proof}
We have 
\[\sum_{d|n}S_1(d)=\sum_{k=1}^n \xi_n^k=\begin{cases} 1& \text{if $n=1$}\\ 0 & \text{if $n>1$}
\end{cases}.\]
The result now follows by the M\"obius inversion formula.
\end{proof}

\begin{lemma}\label{sum}
$\displaystyle S_k(n)=\frac{\phi(n)}{\phi(n/\gcd(k,n))}\mu(n/\gcd(k,n))$.
\end{lemma}

\begin{proof} If $\xi$ is a primitive $n$-th root of $1$ then $\xi^k$ is a primitive
$n/\gcd(n,k)$-th root of $1$. The map $\lambda: u\mapsto u^k$ is a surjective map
from $\Pi_n$ onto $\Pi_{n/\gcd(n,k)}$. We claim that the preimage of every element in 
$\Pi_{n/\gcd(n,k)}$ has the same number of elements. Indeed, if $w_1,w_2$ are
in $\Pi_{n/\gcd(n,k)}$, then $w_2=w_1^a$ for some integer $a$ which can be chosen relatively
prime to $n$. The map $u\mapsto u^a$ gives a bijection between $\lambda^{-1}(w_1)$ and
$\lambda^{-1}(w_2)$. Since $\Pi_m$ has $\phi(m)$ elements for every positive integer $m$,
all 
preimages 
of $\lambda$ have $n_k = \frac{\phi(n)}{\phi(n/\gcd(k,n))}$ elements. Thus
\[S_k(n)=n_k \sum_{w\in \Pi_{n/\gcd(n,k)}}w =
n_k \, \mu(n/\gcd(k,n)).
\]
\end{proof}
From now on we will assume that $n$ is an odd square-free integer, so
$n=p_1p_2\ldots p_t$, where $p_1<p_2<\ldots<p_t$ are odd prime numbers.
In this case, the numbers $\gcd(n,k)$ and $n/\gcd(n,k)$ are relatively prime for any integer $k$.
Using the fact that both $\phi$ and $\mu$ are multiplicative, we have the following

\begin{corollary}
Let $n=p_1p_2\ldots p_t$ be a square-free odd integer. Then
\[S_k(n)=(-1)^t\phi(\gcd(n,k)) \mu(\gcd(k,n)).
\]
\end{corollary}

For the rest of this 
section 
we fix an odd square-free integer $n=p_1p_2\ldots p_t$, where $p_1<p_2<\ldots<p_t$. We assume that $t$ is odd. We write $S_k$ for $S_k(n)$, $\sigma_k$ for
$\sigma_k(n)$. Our key tool will be the following well-known Newton identities:
\[ \sigma_1=-S_1, \ k\sigma_k=-(\sigma_{k-1}S_1+\sigma_{k-2}S_2+\ldots+\sigma_1 S_{k-1}+S_k).\]
(see \cite{gould} for a nice exposition).
From Lemma \ref{sum} we get the following 

\begin{lemma}
If $k<p_1p_2$ then 
\[ S_k=\begin{cases} -1 & \text{if $p_i\nmid k$ for $i=1,\ldots, t$}\\ p_i-1 & \text{if $p_i|k$}.
\end{cases}
\]
\end{lemma}

As a first consequence of the Newton's formulas we have

\begin{lemma}
$\sigma_k=1$ for $1\leq k<p_1$.
\end{lemma}
\begin{proof}
As $S_k=-1$ for $k<p_1$, we have $k\sigma_k=\sigma_{k-1}+\ldots+\sigma_1+1$, so the result
follows by straightforward induction.
\end{proof}

\begin{lemma}
$\sigma_k=0$ for $p_1\leq k<p_2$.
\end{lemma}
\begin{proof}
For $k\geq p_1$ we have 
\[k\sigma_{k}=-(\sigma_{k-1}S_1+\sigma_{k-2}S_2+\ldots+\sigma_{p_1} S_{k-p_1}+S_{k-p_1+1}+\ldots
+S_k).\]

Note that $S_1+\ldots+S_{p_1}=0$. This implies that $\sigma_{p_1}=0$. Furthermore,
\[ (S_{k+1-p_1+1}+S_{k+1-p_1+2}+\ldots+S_{k+1})-(S_{k-p_1+1}+\ldots
+S_k)=S_{k+1}-S_{k+1-p_1}.\]
If $p_1<k+1<p_2$ then either $p_1|k+1$ and then $S_{k+1}=S_{k+1-p_1}=p_1-1$ or $p_1\nmid k+1$
and then $S_{k+1}=S_{k+1-p_1}=-1$. In any case, we conclude that for $p_1\leq k<p_2$
we have 
\[ S_{k-p_1+1}+\ldots+S_k=0.\]
Thus, we have $\sigma_{p_1}=0$ and, for $p_1<k<p_2$,
\[k\sigma_{k}=-(\sigma_{k-1}S_1+\sigma_{k-2}S_2+\ldots+\sigma_{p_1} S_{k-p_1}).\]
The result follows now by straightforward induction.
\end{proof}

Let $r$ be defined by $p_1<p_2<p_3<\ldots<p_r<p_1+p_2<p_{r+1}$. Let $p_2\leq k<p_1+p_2$.
Then $k-p_2<p_1$ so $S_i=-1$ for $1\leq i\leq k-p_2$. Also, $\sigma_i=0$ for $p_1\leq i<p_2$
and $\sigma_i=1$ for $1\leq i<p_1$. It follows that
\[p_2\sigma_{p_2}=-(S_{p_2-p_1+1}+\ldots+S_{p_2})\]
and
\[k\sigma_{k}=(\sigma_{k-1}+\sigma_{k-2}+\ldots+\sigma_{p_2})-(S_{k-p_1+1}+\ldots
+S_k)\]
for $p_1+p_2>k>p_2$.
Among the consecutive integers $p_2-p_1+1,\ldots, p_2$ exactly one is divisible by $p_1$, exactly one is divisible
by $p_2$ and none of the remaining numbers is divisible by any $p_i$.  It follows that
$S_{p_2-p_1+1}+\ldots+S_{p_2}=p_2-1+p_1-1+(p_1-2)(-1)=p_2$ and $\sigma_{p_2}=-1$.
If $p_1+p_2>k+1>p_2$, then subtracting the equalities
\[k\sigma_{k}=(\sigma_{k-1}+\sigma_{k-2}+\ldots+\sigma_{p_2})-(S_{k-p_1+1}+\ldots
+S_k)\]
and 
\[(k+1)\sigma_{k+1}=(\sigma_{k}+\sigma_{k-1}+\ldots+\sigma_{p_2})-(S_{k+1-p_1+1}+\ldots
+S_{k+1})\]
we get
\[(k+1)(\sigma_{k+1}-\sigma_k)=S_{k+1-p_1}-S_{k+1}.\]
If none of the $p_i$ divides $k+1$, then it also does not divide $k+1-p_1$ so $S_{k+1-p_1}=S_{k+1}=-1$. It follows that $\sigma_{k+1}=\sigma_{k}$ in this case.
If $p_1|k+1$ then $p_1|p+1-p_1$ so $S_{k+1-p_1}=S_{k+1}=p_1-1$. Again, $\sigma_{k+1}=\sigma_{k}$ in this case. Finally, if $p_i$ divides $k+1$ then $k+1=p_i$, $S_{k+1-p_1}=-1$, $S_{k+1}=p_i-1$
and $\sigma_{k+1}=\sigma_{k}-1$. We get the following result:

\begin{theorem}
For $p_i\leq k< p_{i+1}$ we have $\sigma_k=-k+1$, $i=2,\ldots, r-1$ and for $p_r\leq k<p_1+p_2$
we have $\sigma_k=-r+1$.
\end{theorem} 

Since $\Phi_{2n}(x)=\Phi_n(-x)$ for $n$ odd we get the following

\begin{theorem}\label{main}
If $n=p_1\ldots p_k$ where $p_i$ are odd primes and
$p_1<p_2<\ldots<p_r<p_1+p_2<p_{r+1}<\ldots<p_t$ with $t\geq 3$ odd,
then the numbers $-(r-2),-(r-3),\ldots, r-2, r-1$ are all coefficients
of $\Phi_{2n}$. Furthermore, if $1+p_r<p_1+p_2$ then $1-r$ is also a coefficient of $\Phi_{2n}$.
\end{theorem}

\section{Conjectural symmetry in the distribution of nontrivial coefficients}
In this section we share some computer generated evidence that suggests some sort of symmetry between
the appearance of positive and negative coefficients of cyclotomic polynomials. We attempt to formalize these observations in Conjectures \ref{conj1} and \ref{conj2}.


\subsection{Asymptotic symmetry of first appearances of nontrivial coefficients}\label{first}
{\sffamily{Let $A$ be the set of points $(c,n)$ such that $c$ is a nontrivial coefficient of $\Phi_n$, and $c$ is not a coefficient of any $\Phi_m$ for any $m<n$.}} We enumerate (by consecutive positive integers) the points of $A$ as follows: the point with the smaller value of $n$ is enumerated before the point with the larger value of $n$; if two points of $A$ have the same value of $n$, then the point with the smaller value of $c$ is enumerated first. {\sffamily{Let $A_k$ be a subset of $A$ consisting of the first $k$ points of $A$.}} Brett Haines obtained the following graphs with the help of Wolfram Mathematica \cite{wolf}. 

\begin{enumerate}
	\item Graph of $N_{100}$.
	\begin{center}
		\includegraphics[scale=.3]{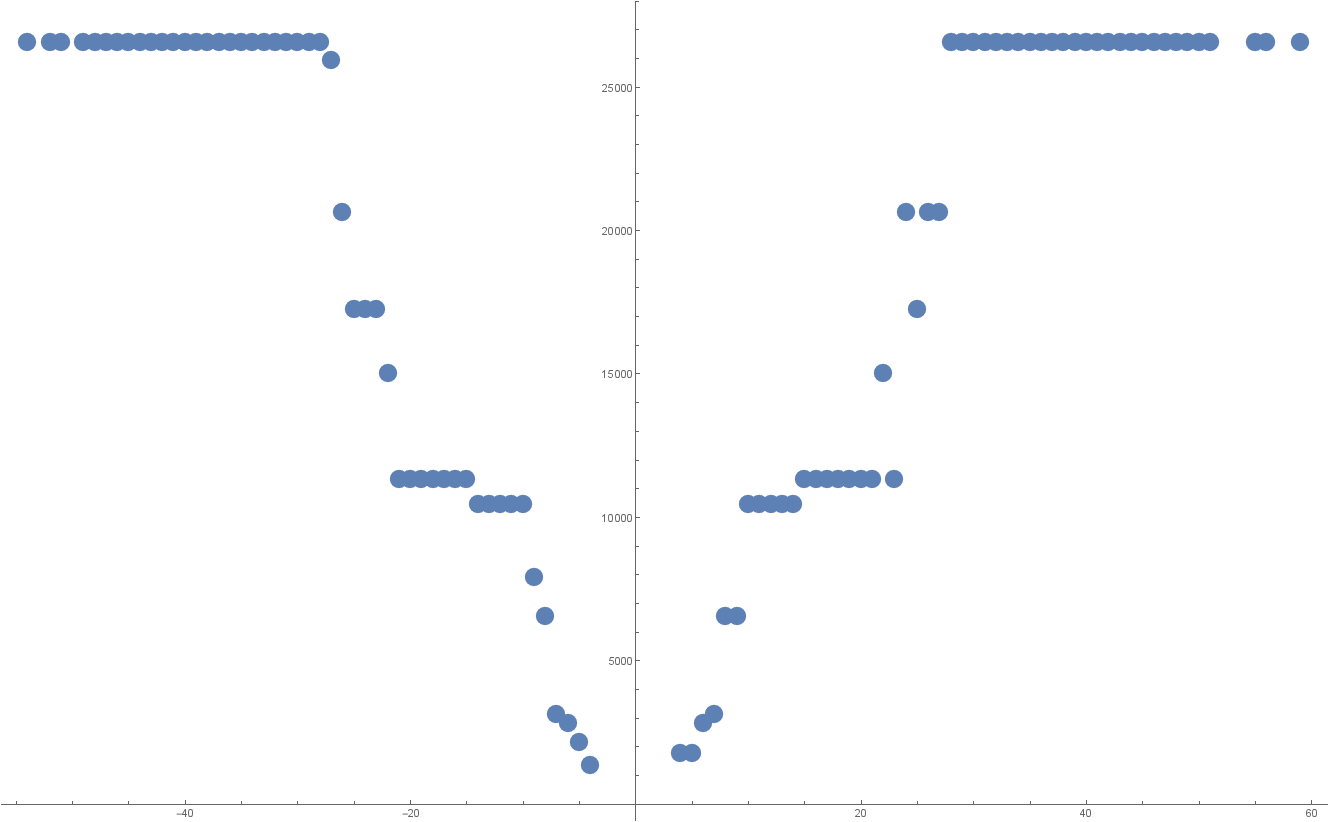}
	\end{center}
	\item Graph of $A_{250}$.
	\begin{center}
		\includegraphics[scale=.3]{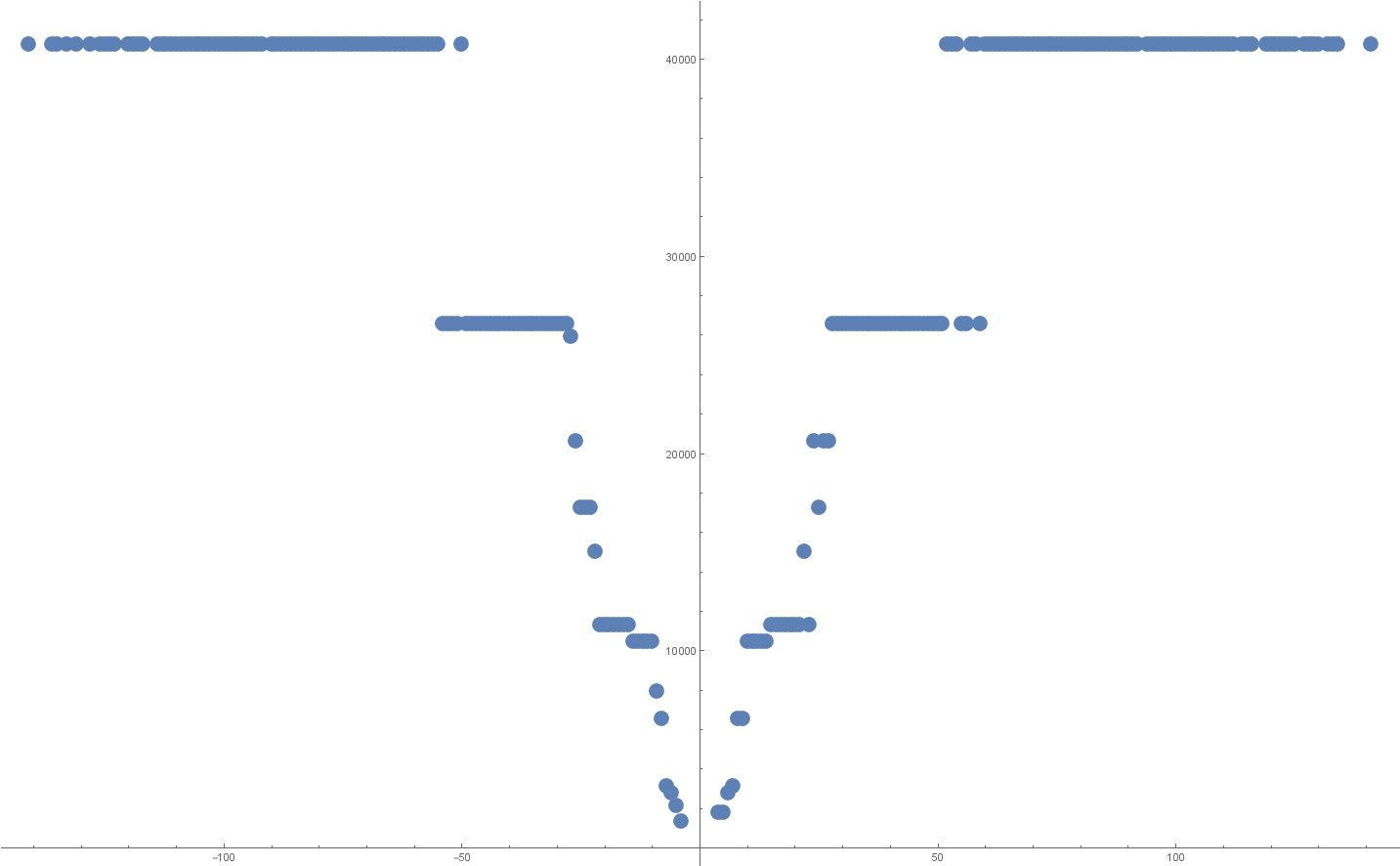}
	\end{center}
	\item Graph of $A_{1000}$.
	\begin{center}
		\includegraphics[scale=.3]{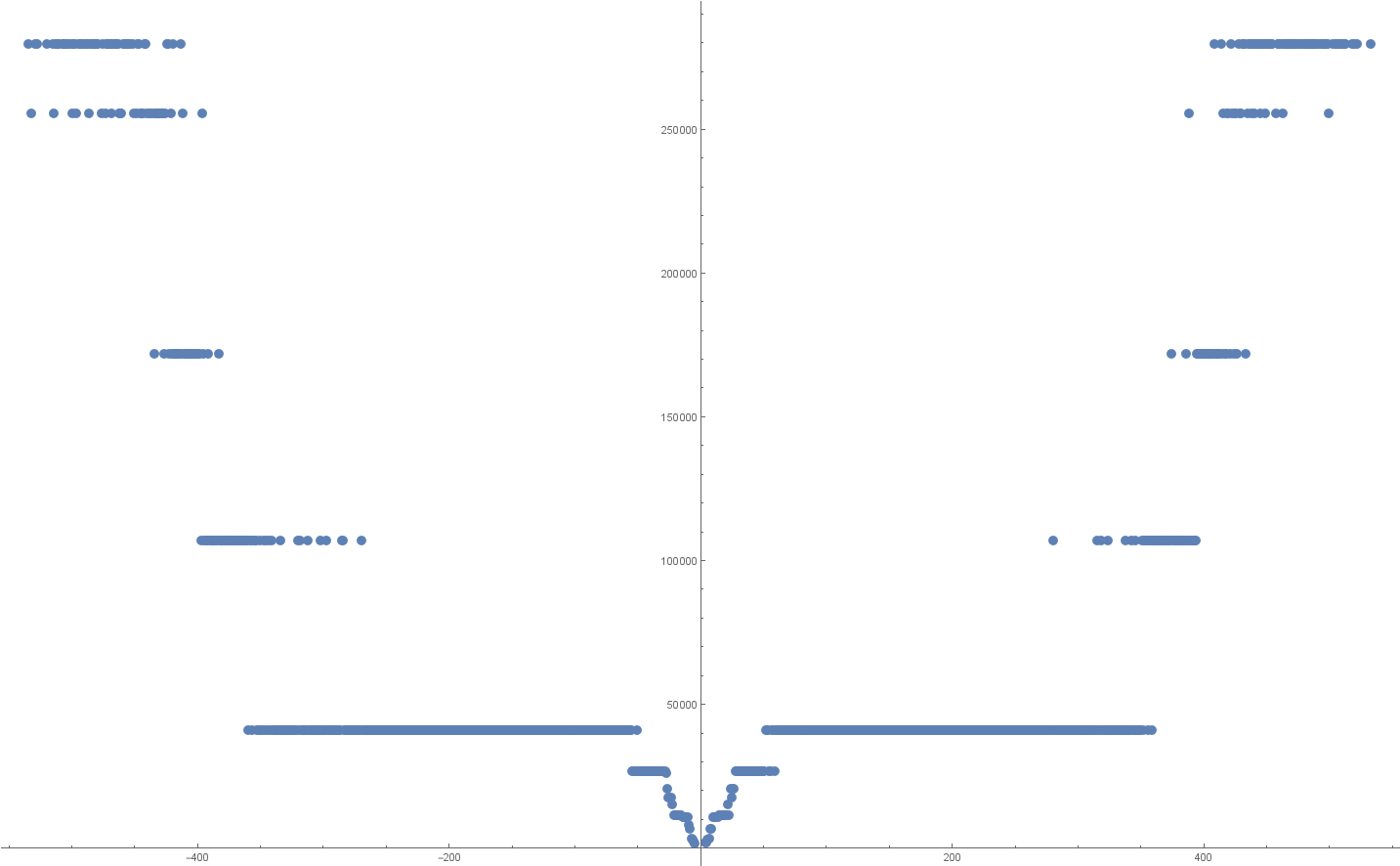}
	\end{center}
	\item Graph of $A_{10000}$.
	\begin{center}
		\includegraphics[scale=.3]{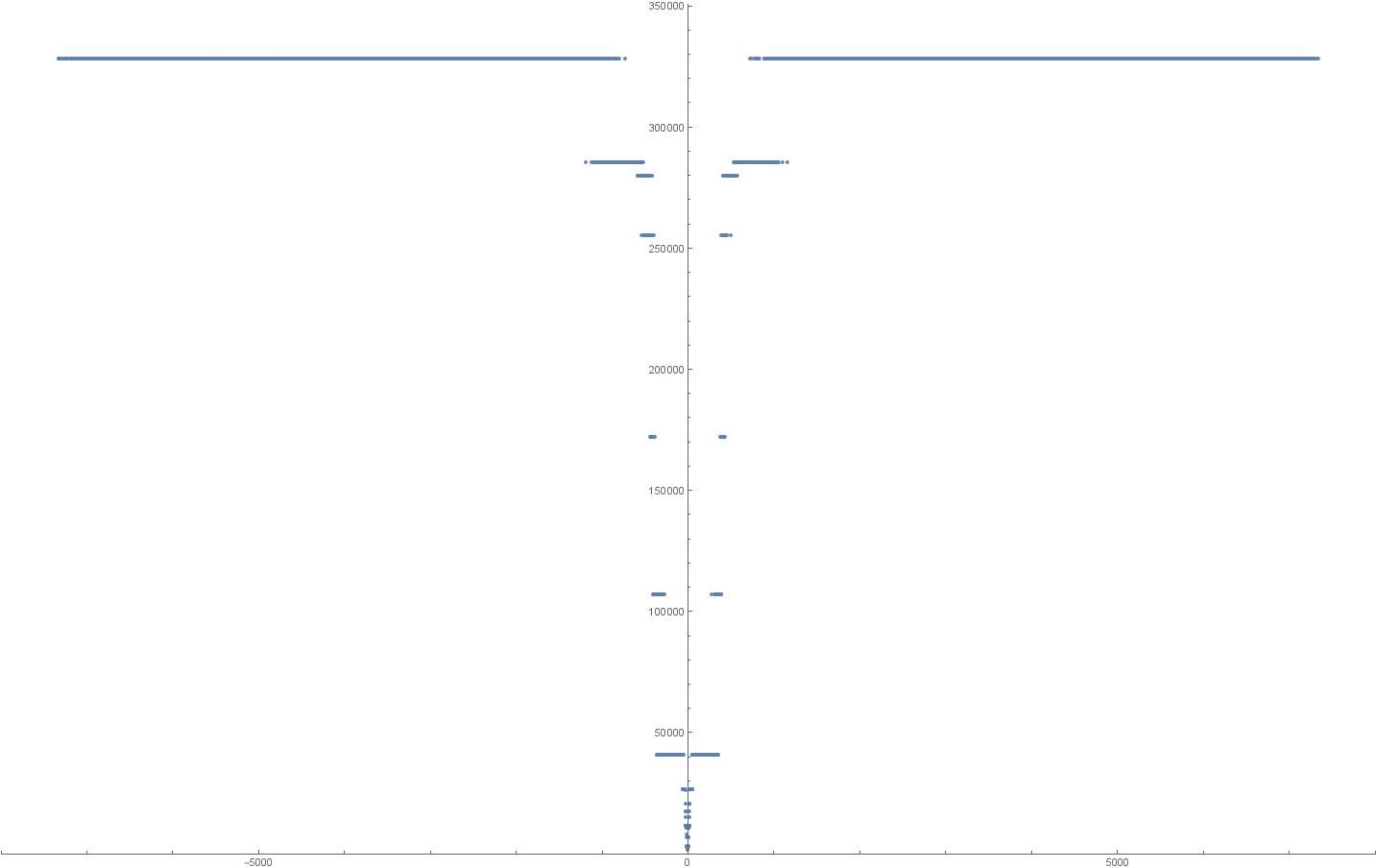}
	\end{center}
\end{enumerate}

\subsection{Asymptotic symmetry of nontrivial coefficients}\label{total} 
{\sffamily{Let
$B$ be the set of points $(c,n)$ such that that $c$ is a nontrivial coefficient of $\Phi_n$. Let $B_k$ consist of all points $(c,n)$
in $B$ such that $n \le k$.}} William Tyler Reynolds obtained the following graphs with the help of SAGE \cite{sage}.

\begin{enumerate}
	\item Graph of $B_{1000}$.
	$$\includegraphics[scale=.4]{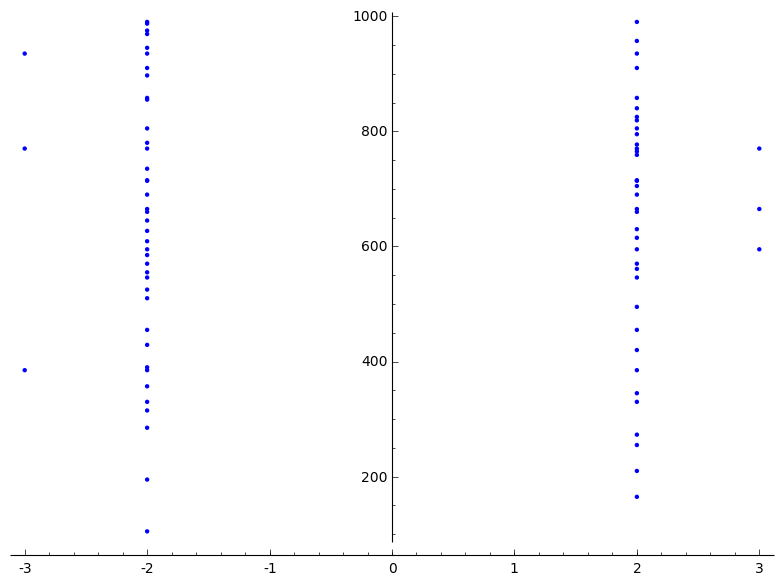}$$
	\item Graph of $B_{10000}$.
	$$\includegraphics[scale=.4]{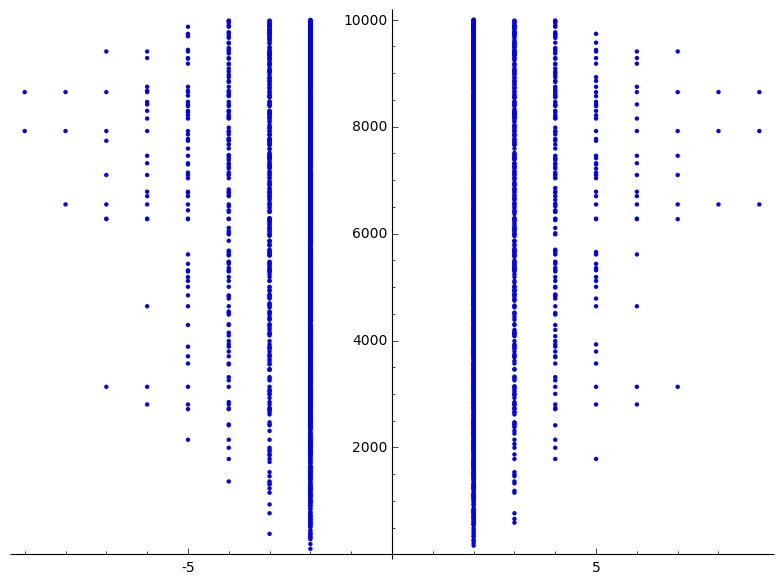}$$
	\item Graph of $B_{250000}$.
	$$\includegraphics[scale=.4]{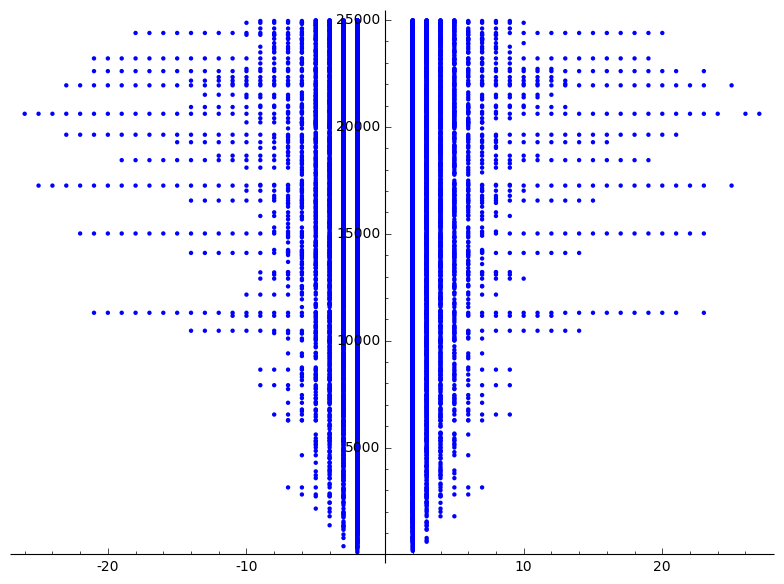}$$
	\item Graph of $B_{500000}$.
	$$\includegraphics[scale=.4]{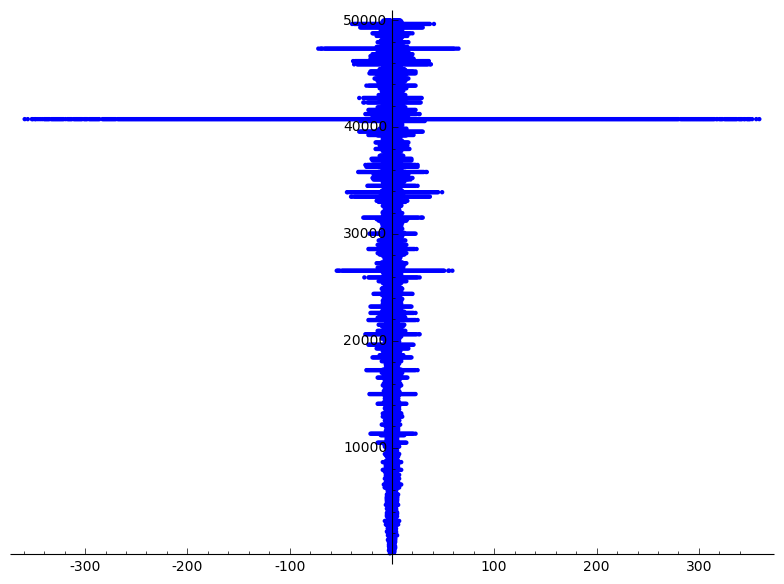}$$
\end{enumerate}

\subsection{Conjectures}\label{conj}
The eight pictures displayed in \ref{first} and \ref{total} above are the graphs of some four ascending finite subsets of each of the sets $A$ and $B$.
Let $A^+_k$ be the subset of $A_k$ consisting of those points $(c,n)$ with $c>0$;  
let $A^-_k$ be the subset of $A^-_k$ consisting of those points $(c,n)$ with $c<0$. The sets $B^+_k$ and $B^-_k$ are defined in a similar way. 
{\sffamily{We want to formalize the observation that the graphs of the sets $A^+_k$ and $A^-_k$ on a fixed computer screen appear increasingly more symmetric across the vertical $n$-axis as $k \to \infty$.}} 

If $S$ is a finite subset of the $c,n$-plane that lies in the upper half plane, then by ${S}^{\sharp}$ we denote the reflection of $S$ across the $n$-axis. If $c',n' > 0$, then we define $[c',n']S$ as the set $\{(c/c',n/n'): (c,n) \in S\}$, which is a subset of the unit square $[0,1]^2$.
Let $\mathcal H$ be the Hausdorff distance in $[0,1]^2$ induced by the standard Euclidean distance, and let $|S|$ denote the number of points in $S$. 

\begin{conjecture}\label{conj1}
	For any positive integer $k$, let $[-c_k,c_k] \times [0,n_k]$ be the smallest rectangle containing $A_k$. Then $A^+_k$ contains a subset $L_k$, and $A^-_k$ contains a subset $M_k$, such that 
	$\d \lim_{k \to \infty} \mathcal H \left([c_k,n_k]L_k,{[c_k,n_k]M_k}^{\sharp} \right) = 0 $ and $\d \lim_{k \to \infty} |A^+_k| / | L_k|  = \lim_{k \to \infty} |A^-_k| / | M_k| = 1$.
\end{conjecture}

\begin{conjecture}\label{conj2}
	We make a similar conjecture for the family $\{B_k\}$.
\end{conjecture}

We hope that we have formulated the weakest possible conjectures that support our symmetry claims. Stronger results may be true.

\end{document}